\documentclass[12pt,a4paper,leqno]{amsart}
\usepackage{amssymb,xspace}
\usepackage{amstext}
\usepackage{pdfsync}
\usepackage{hyperref}
\usepackage[mathscr]{eucal}
\theoremstyle{plain}
\usepackage{amsbsy,amssymb,amsfonts,latexsym,eucal,amscd}
\usepackage[dvips]{graphicx}
\usepackage{epsfig}
\usepackage[all]{xy}
\usepackage{pstricks}
\usepackage{mathrsfs}
\usepackage{tikz}
\usetikzlibrary{shapes}
\usetikzlibrary{positioning}
\usetikzlibrary{arrows}
\usetikzlibrary{calc}
\usepackage{txfonts}
\usepackage{hyperref}
\makeatletter
\def\l@subsection{\@tocline{2}{0pt}{1pc}{5pc}{\hspace{2em}}}
\makeatother

\newcommand{\tens}[1][]{\mathbin{\otimes_{\raise1.5ex\hbox to-.1em{}{#1}}}}
\newcommand{\lltens}[1][]{{\mathop{\tens}\limits^{\rm \M{L}}}_{#1}}

\newcommand{\C }{\ensuremath{\mathbb C}}
\newcommand{\Q }{\ensuremath{\mathbb Q}}
\newcommand{\Z }{\ensuremath{\mathbb Z}}

\newcommand{\oo }{\ensuremath{\mathcal{O}}}

\newtheorem{theorem}{Theorem}[section]

\newtheorem{proposition}[theorem]{Proposition}

\theoremstyle{plain}
\newtheorem*{Theorem}{Theorem}

{\theoremstyle{definition}}

{\theoremstyle{definition}}

{\theoremstyle{definition}}

{\theoremstyle{definition}}

{\theoremstyle{definition}}

{\theoremstyle{definition}}

{\theoremstyle{definition}\newtheorem{remark}[theorem]{Remark}}

\newtheorem{conjecture}[theorem]{Conjecture}

\newcommand{\M}[1]{\mathrm{#1}}

\entrymodifiers={+!!<0pt,\fontdimen22\textfont2>}

\def\apl#1#2#3{#1\mkern -1 mu:\mkern - 6 mu
\xymatrix@C=17pt{#2\!\ar[r]&\!#3}
}

\def\aplexp#1#2#3#4{#1\mkern -1 mu:\mkern - 6 mu
\xymatrix@C=17pt{#2\!\ar[r]^-{#4}&\!#3}
}

\def\aplcourte#1#2#3{#1\mkern -4 mu:\mkern - 8 mu
\xymatrix@C=12pt{#2\!\ar[r]&\!#3}
}

\def\aplpt#1#2#3#4{#1\mkern -4 mu:\mkern - 8 mu
\xymatrix@C=17pt{#2\!\ar[r]&\!#3#4}
}
\author{Julien Grivaux}

\address{CNRS, LATP\\
UMR 6632\\
CMI, Universit\'{e} de Provence\\
39, rue Fr\'{e}d\'{e}ric Joliot-Curie\\
13453 Marseille Cedex 13\\
France.}

\email{jgrivaux@cmi.univ-mrs.fr}
\title{Formality of derived intersections}

\begin{document}

\begin{abstract}
We study derived intersections of smooth analytic cycles, and provide in some cases necessary and sufficient conditions for this intersection be formal. In particular, if $X$ is a complex submanifold of a complex manifold $Y$, we prove 
that $X$ can be quantized if and only if the derived intersection of $X^2$ and $\Delta_Y$ is formal in $\mathrm{D}^{\mathrm{b}}\bigl (X^2 \bigr)$. 
\end{abstract}
\vspace*{1.cm}
\maketitle
\tableofcontents

\section{Introduction}
The theory of non-transverse intersections is now a classical topic in algebraic geometry, which is exposed in the classical book  \cite{FU}. The main idea goes as follows: if two (smooth) subschemes $X$ and $Y$ of a smooth complex algebraic scheme $Z$ intersect along an irreducible scheme $T$, then the refined intersection of $X$ and $Y$ consists of the two following data: the set-theoretic intersection $T=X \cap Y$, and an element $\xi$ of the Chow group $\M{CH}^r (T)$ where $r$ is the difference between the expected codimension of $X \cap Y$ (that is the sum of the codimensions of $X$ and $Y$) and the true codimension of $X \cap Y$. If $\widehat{N}$ is the restriction of $N_{X/Z}$ to $T$, the conormal excess bundle $\mathcal{E}$ is an algebraic vector bundle of rank $r$ defined by the exact sequence
\[
0 \longrightarrow \mathcal{E} \longrightarrow \widehat{N}^* \longrightarrow N^*_{T/Y} \longrightarrow 0.
\]
Then the excess intersection class $\xi$ is given by $\xi=\M{c}_r \,(\mathcal{E}^*)$. This formula is a particular case of the excess intersection formula for refined Gysin morphisms, we refer the reader to \cite[\S 6.2 and \S6.3]{FU} for more details. \par \bigskip
Recall now that  
\[
\M{ch} \, \colon \, \M{K}(Z) \otimes_{\Z} \Q \longrightarrow \M{CH}(Z) \otimes_{\Z} \Q
\]
is an isomorphism, so that dealing with Chow rings is the same as dealing with $K$-theory (modulo torsion). Intersection
theory in $K$-theory is much more easier and goes back to \cite{BS}: the product of $\oo_X$ and $\oo_Y$ is just the alternate sum of the classes of the sheaves $\M{Tor}^{\,i}_{\oo_Z}(\oo_X, \oo_Y)$, which is the element
\[
\Lambda_{-1}\,\mathcal{E} := \sum_{i \geq 0} (-1)^i \, \left[\Lambda^i \, \mathcal{E} \right].
\]
Then $\M{ch}_r\left( \Lambda_{-1}\, \mathcal{E} \right)$ is exactly the excess class $\xi$ introduced above. For a geometric interpretation of intersection in $\M{K}$-theory, we refer the reader to the classical Tor formula \cite[\S 5.C.1]{SE}.
\par \bigskip
For applications, it is essential to lift $\M{K}$-theoretic results at the level of derived categories. The intersection product of two algebraic cycles $X$ and $Y$ in this framework is simply the derived tensor product ${\vartheta}:=\oo_X \, \lltens[\oo_Z]\, \oo_Y$ in $\M{D}^{\M{b}}(Z)$, which is much more complicated to study. First of all, the element $\vartheta$ admits canonical lifts in $\M{D}^{\M{b}}(X)$ or $\M{D}^{\M{b}}(Y)$, but does not always comes from an object in $\M{D}^{\M{b}}(T)$. Thus $\vartheta$ doesn't live on the set-theoretic intersection of $X$ and $Y$. Next, for any nonnegative integer $p$, there is a canonical isomorphism between $\mathcal{H}^{-p}(\vartheta)$ and $\Lambda^p \, \mathcal{E}$. There is therefore a natural candidate for $\vartheta$, namely the formal object $\mathfrak{s} \left( \mathcal{E} \right):= \bigoplus_{i=0}^d \Lambda^i \mathcal{E}\, [i]$. It is a rather delicate question to decide if the object $\vartheta$ is formal in $\M{D}^{\M{b}}(X)$, $\M{D}^{\M{b}}(Y)$ or in $\M{D}^{\M{b}}(Z)$. The first important result in this direction was given by Arinkin and C\u{a}ld\u{a}raru for self-intersections. Their result is a far-reaching generalisation of the Hochschild-Kostant-Rosenberg isomorphism corresponding to the self-intersection of the diagonal in $X \times X$. Let us state it in the analytic category\footnote{Since we are concerned with applications in complex geometry, we will state throughout the paper the results in the analytic category; they are also valid in the algebraic category as well.}:
\begin{Theorem}[\cite{AC}] Let $X$, $Y$ be smooth complex manifolds such that $X$ is a closed complex submanifold of $Y$. Then $j_{X/Y}^* \, j_{X/Y*}\, \oo_X$ is formal in $\M{D}^{\M{b}}(X)$ if and only if $N^*_{X/Y}$ can be lifted to a holomorphic vector bundle on the first formal neighbourhood of $X$ in $Y$.
\end{Theorem} 
In particular, if the injection of $X$ into its first formal neighbourhood in $Y$ admits a retraction $\sigma$, which means in the terminology of \cite{HKR} that $(X, \sigma)$ is a quantized analytic cycle, then $j_{X/Y}^* \, j_{X/Y*}\, \oo_X$ is formal. This fact was discovered by Kashiwara in the beginning of the nineties \cite{KA}, and his proof has been written down in \cite{HKR}. 
In this paper, we are interested in more general intersections than self-intersections. We will deal with cycles intersecting linearly, which means that the intersection is locally biholomorphic to an intersection of linear subspaces. We will also make the additional assumption that one of the cycles, say $X$, can be quantized. Our main result runs as follows \footnote{The author was informed that a similar result has been announced independently by Arinkin, C\u{a}ld\u{a}raru and Hablicsek \cite{ACH}.}:
\begin{theorem} \label{boss} Let $Z$ be a complex manifold, let $(X, \sigma)$ be a quantized analytic cycle in $X$, and $Y$ be a smooth complex submanifold of $Z$ such that $X$ and $Y$ intersect linearly. Put $T=X \cap Y$, and let $\widehat{N}$ be the restriction of $N_{X/Z}$ to $T$, and let $\mathcal{E}$ be the associated conormal excess bundle on $T$. Then the object $j_{Y/Z}^* \,\, j_{X/Z *} \,  \mathcal{O}_X$ is formal in the derived category $\M{D}^{\M{b}}(Y)$ if and only if and only if the excess conormal exact sequence
\[
0 \longrightarrow \mathcal{E} \longrightarrow \widehat{N}^* \longrightarrow N^*_{T/Y} \longrightarrow 0
\]
is holomorphically split.
\end{theorem}
Two extreme cases of this theorem are already known: the case of self-intersections and the case of transverse intersections. For self-intersections, the formality follows directly from Arinkin-Cald\u{a}r\u{a}ru's result, since $X$ is assumed to be quantized. For transverse intersections, there are no higher $\M{Tor}$ sheaves so that $j_{Y/Z}^* \,\, j_{X/Z *} \,  \mathcal{O}_X\simeq j_{T/Y*}\,\oo_T$. An important point must be noticed: although $j_{Y/Z}^* \,\, j_{X/Z *} \,  \mathcal{O}_X$ is locally formal, and hence comes locally from $\M{D}^{\M{b}}(T)$, this property is no longer true globally in general if the excess exact sequence is not split; we provide counter-examples. 
\par \medskip
One of the principal applications of Theorem \ref{boss} is the existence of a formality criterion in order that an analytic cycle be quantized:
\begin{theorem} \label{QC}
Let $X$ be a smooth complex submanifold of a complex manifold $Y$. Then $X$ can be quantized if and only if the object
$j_{X^2/Y^2}^* \, \mathcal{O}_{\Delta_Y}$ is formal in $\M{D}^{\M{b}}({X^2})$. 
\end{theorem}
This result follows directly from Theorem \ref{boss} by applying the trick of the reduction to the diagonal. Thus, the existence of an extension to the conormal bundle at the first order or the existence of an infinitesimal retraction for a complex analytic cycle are both equivalent to the formality of a derived intersection: the first one involves the pair $(X, X)$ in $Y$ and the second one involves the pair $\bigl(\Delta_{Y}, X^2\bigr)$ in $Y^2$. However, it would be very interesting to know if $\oo_X \, \lltens[\oo_Y] \, \oo_X$ is always formal in $\M{D}^{\M{b}}(Y)$, since no necessary condition for the formality of this complex is known up the the author's knowledge. We conjecture that it is indeed the case, and provide some evidence towards this conjecture.
\par \medskip
\textbf{Acknowledgments.} This paper grew up from conversations with Daniel Huybrechts and Richard Thomas, I would like to thank them both for many interesting discussions on this topic. 
\section{Notations and conventions}
\par \smallskip
All complex manifolds considered are smooth and connected. By submanifold, we always mean \textit{closed} submanifold. If $X$ and $Y$ are complex analytic spaces such that $X$ is a complex subspace of $Y$, we denote by $j_{X/Y}$ the corresponding inclusion morphism, and we denote by $\overline{X}_Y$ the first formal neighbourhood of $X$ in $Y$. For any holomorphic vector bundle $\mathcal{E}$ on a complex manifold $X$, we denote by $\mathfrak{s} \left( \mathcal{E} \right)$ the formal complex $\bigoplus_{k \geq 0} \Lambda^k \mathcal{E} \,[k]$. 
\par \medskip
For any ringed complex space $(X, \oo_X)$, we denote by $\M{D}^{\M{b}}(X)$ the bounded derived category of the abelian category of coherent sheaves of $\oo_X$-modules on $X$. For any holomorphic map $f \colon X \longrightarrow Y$ between complex manifolds, we denote by $f^* \colon \M{D}^{\M{b}}(Y) \longrightarrow  \M{D}^{\M{b}}(X)$ the associated derived pullback. The non-derived pullback functor is denoted by $\M{H}^0 \left(f^* \right)$.

\section{Quantized analytic cycles and AK complexes} \label{AK}
We recall terminology and basics constructions from \cite{HKR}. Let $Z$ be complex manifold. A \textit{quantized analytic cycle} in $Z$ is a pair $(X, \sigma)$ such that $X$ is a closed connected complex submanifold of $Z$ and $\sigma$ is a retraction of the injection of $X$ into its first formal neighbourhood in $Z$. If we denote the latter by $\overline{X}_Z$, then $\sigma$ is a morphism of sheaves of $\mathbb{C}_X$-algebras from $\oo_X$ to $\oo_{\overline{X}_Z}$. Equivalently, $\sigma$ is a splitting of the Atiyah sequence
\begin{equation} \label{atiyah}
0 \longrightarrow N^*_{X/Z}\longrightarrow \oo_{\overline{X}_Z} \longrightarrow \oo_X \longrightarrow 0
\end{equation}
in the category of sheaves of $\C_X$-algebras, that is an isomorphism of sheaves of $\C_X$-algebras between $\oo_{\overline{X}_Z}$ and the trivial extension of $\mathcal{O}_X$ by $N^*_{X/Z}$. The set of possible quantizations $\sigma$ of $X$ is either empty or an affine space over $\M{Hom} \left(\Omega_X^1, N^*_{X/Z}\right)$. For any quantized analytic cycle $(X, \sigma)$ in $Z$, there is a canonical complex $\mathcal{P}_{\sigma^{\vphantom{a}}}$ of $\oo_{\overline{X}_Z}$-modules, called the \textit{Atiyah-Kashiwara complex}, defined as follows: for any nonnegative integer $k$,
\[
(\mathcal{P}_{\sigma^{\vphantom{a}}})_{-k}=\Lambda^k_{\oo_X} \oo_{\overline{X}_Z}=\Lambda^{k+1} N^*_{X/Z} \oplus \Lambda^{k} N^*_{X/Z}.
\] Besides, $(\mathcal{P}_{\sigma^{\vphantom{a}}})_{-k}$ is acted on by $\oo_{\overline{X}_Z}$ via $\sigma$ as follows: $\mathcal{O}_X$ acts in the natural way by multiplication, and $N^*_{X/Z}$ acts on $\Lambda^{k+1} N^*_{X/Z}$ by $0$ and on $\Lambda^{k} N^*_{X/Z}$ by the formula 
\[
i \, . \, (i_1 \wedge \ldots \wedge i_k)=i \wedge i_1 \wedge \ldots \wedge i_k.
\] The differential $\frac{1}{k} d_{-k}$ is given by the composition 
\[
\Lambda^{k+1} N^*_{X/Z} \oplus \Lambda^{k} N^*_{X/Z} \xrightarrow{(0, \,\M{id})} \Lambda^{k} N^*_{X/Z} \xrightarrow{(\M{id}, \,0)} \Lambda^{k} N^*_{X/Z}\oplus \Lambda^{k-1} N^*_{X/Z}.
\]
The complex $\mathcal{P}_{\sigma^{\vphantom{a}}}$ is a left resolution of $\mathcal{O}_X$ over $\oo_{\overline{X}_Z}$ which neither projective nor flat. However, the morphism
\[
j_{X/Z}^* \,\, j_{X/Z*} \, \oo_X \simeq j_{X/Z}^* \, \mathcal{P}_{\sigma^{\vphantom{a}}} \longrightarrow \mathcal{P}_{\sigma^{\vphantom{a}}} \otimes_{\mathcal{O}_{\overline{X}_Z}} \mathcal{O}_X=\mathfrak{s}\left( N^*_{X/Z}\right)
\]
is an isomorphism \cite[Proposition 4.3]{HKR}, called the \textit{quantized Hochschild-Kostant-Rosenberg isomorphism.}
If $\sigma$ is fixed, so that the isomorphism between $\oo_{\overline{X}_Z}$ and $N^*_{X/Z} \oplus \,\oo_X$ is fixed, and if $\varphi$ is in $\M{Hom} \left(\Omega_X^1, N^*_{X/Z} \right)$, then $\mathcal{P}_{\sigma+\varphi}=\Lambda^{k+1} N^*_{X/Z} \oplus \Lambda^{k} N^*_{X/Z}$ where $\oo_X$ acts on $\Lambda^{k+1} N^*_{X/Z}$ by multiplication, on $\Lambda^k N^*_{X/Z}$ by $f \, . \, (i_1 \wedge \ldots \wedge i_k)=\varphi(df) \wedge i_1 \wedge \ldots \wedge i_k$, and $N^*_{X/Z}$ acts by wedge product. If $N^*_{X/Z}$ admits a holomorphic connection $\nabla \colon N^*_{X/Z} \longrightarrow  \Omega_X^1 \otimes N^*_{X/Z}$, there is an associated natural isomorphism between $\mathcal{P}_{\sigma^{\vphantom{a}}}$ and $\mathcal{P}_{\sigma + \varphi}$ given in any degree $-k$ by 
\begin{equation} \label{change}
(\underline{i}, \underline{j}) \rightarrow (\underline{i}+ \{ \varphi \wedge \Lambda^k \nabla \}(\underline{j}), \underline{j}).
\end{equation}
\section{Linear intersections and excess contributions} \label{linear}
Let us consider three complex manifolds $X$, $Y$ and $Z$ such that $X$ and $Y$ are closed complex submanifolds of $Z$. Let $T=X \cap Y$. We will say that $X$ and $Y$ intersect \textit{linearly} if for any point $t$ in $T$, there exist local holomorphic coordinates of $Z$ near $t$ such that $X$ and $Y$ are given by linear equations in these coordinates. If this is the case, $T$ is smooth, and there exists a holomorphic vector bundle $\mathcal{E}$ on $T$ fitting in the exact sequence
\begin{equation} \label{excess}
0 \longrightarrow \mathcal{E} \overset{\alpha}{\longrightarrow} \widehat{N}^* \overset{\pi}{\longrightarrow} N^*_{T/Y} \longrightarrow 0
\end{equation}
where $\widehat{N}=j_{T/X}^*\, N_{X/Z}$. The bundle $\mathcal{E}$ is called the \textit{excess intersection bundle}. If two cycles $X$ and $Y$ in $Z$ intersect linearly then for any nonnegative integer $k$, $\M{Tor}^k_{\oo_Z} (\oo_X, \oo_Y)=\Lambda^k \mathcal{E}$. The two extreme cases of linear intersections are transverse intersections (in that case $\mathcal{E}= \{0 \}$) and self-intersections (in that case $X=Y$ and $\mathcal{E}=N^*_{X/Z}$). 
\par \medskip
We give computations in local coordinates which will be useful in the sequel. By definition of linear intersections,
we can find local holomorphic coordinates 
\[
\left\{(x_i)_{1 \leq i \leq p}, (y_j)_{1 \leq j \leq q}, (t_k)_{1 \leq k \leq r}, (z_l)_{1 \leq l \leq s }\right\}
\] on $Z$ in a neighbourhood $U$ of any point in $T$ such that locally
\begin{alignat*}{5}
&X  =\, &&\{x_1=0\} &&\cap \ldots \cap \{x_p=0\} &&\cap \{t_1=0\} \cap \ldots \cap \{t_l=0\}
\\
&Y =\,  &&\{y_1=0\} &&\cap \ldots \cap \{y_q=0\} &&\cap \{t_1=0\} \cap \ldots \cap \{t_l=0\}
\end{alignat*}
The excess bundle $\mathcal{E}$ is the bundle spanned by $dt_1, \ldots, dt_r$.
Let $\mathcal{K}[U]$ be the Koszul complex on $U$ associated with the regular sequence \[
x_1, \ldots, x_p, t_1, \ldots, t_r\]
defining $X$, and let $\tau$ be the local infinitesimal retraction of $\overline{X}_Z$ in $X$ induced by the projection $(\underline{x}, \underline{y}, \underline{t}, \underline{z}) \rightarrow (\underline{0}, \underline{y}, \underline{0}, \underline{z})$.
Then there is a natural morphism of complexes $ \gamma$ from $\mathcal{K}(U)$ to $\left(\mathcal{P}_{\tau}\right)_{\, |\,U}$ given as follows (see the proof of \cite[Proposition 3.2]{HKR}): if $I \subset\{ 1 \ldots, p \}$ and $J \subset \{ 1 \ldots, r \}$, then
\begin{align}  \label{morphisme}
\gamma \left[  f(\underline{x}, \underline{y}, \underline{t}, \underline{z}) \,\, dx_{I} \wedge dt_{J} \right]=&\left( \sum_{a=1}^p \frac{\partial f}{\partial x_a}(\underline{0}, \underline{y}, \underline{0}, \underline{z}) dx_a+ \sum_{b=1}^r \frac{\partial f}{\partial t_b}(\underline{0}, \underline{y}, \underline{0}, \underline{z}) dt_b \right) \wedge dx_{I} \wedge dt_{J} \notag \\
& \hspace{2cm}\oplus f(\underline{0}, \underline{y}, \underline{0}, \underline{z}) \,\wedge dx_{I} \wedge  dt_{J}.
\end{align}

\section{Restriction of the AK complex} \label{cart}
We keep the notation of \S \ref{linear}. The cartesian diagram of analytic spaces
\[
\xymatrix {
\overline{T}_Y \ar@{}[dr]|{\square} \ar[r] \ar[d]& \overline{X}_Z \ar[d] \\
Y \ar[r] & Z
}
\]
yields a canonical morphism from $\left(\oo_{\overline{X}_Z}\right)_{\, | Y}$ to $\oo_{\overline{T}_Y}$ which is in fact an isomorphism (this can be checked using local holomorphic coordinates as in the end of \S \ref{excess}). If we restrict the sequence \ref{atiyah} to $Y$, we get the exact sequence
\[
\M{Tor}^1_{\oo_Z}(\oo_X, \oo_Y) \longrightarrow \widehat{N} \longrightarrow  \left(\oo_{\overline{X}_Z}\right)_{\, |Y} \longrightarrow \oo_T \longrightarrow 0
\]
which is exactly the sequence
\[
0 \longrightarrow \mathcal{E} \longrightarrow  \widehat{N} \longrightarrow \oo_{\overline{T}_Y} \longrightarrow \oo_T \longrightarrow 0
\]
obtained by merging the sequences (\ref{atiyah}) and (\ref{excess}) corresponding to the pair $(T, Y)$. 
\par \bigskip
We turn now to the quantized situation, where we can describe explicitly the  restriction to $Y$ functor from $\M{Mod} \left(\oo_{\overline{X}_Z} \right)$ to $\M{Mod} \left(\oo_{\overline{T}_Y} \right)$. Assume that $(X, \sigma)$ is a quantized analytic cycle in $Z$. Let us define a sheaf of rings $\widehat{\oo}_T$ on  $T$ by the formula $\widehat{\oo}_T= \left(\sigma_* \oo_{\overline{X}_Z}\right)_{|\,T}$. Then we have an exact sequence
\begin{equation} \label{malin}
0 \longrightarrow \mathcal{E} \longrightarrow \widehat{\oo}_T  \longrightarrow \oo_{\overline{T}_Y} \longrightarrow 0.
\end{equation}
The inclusion morphism from $\overline{T}_Y$ to $\overline{X}_Z$ can be described as follows: $\oo_{\overline{X}_Z}$ is an $\mathcal{O}_X$-module via $\sigma$. Since $N^*_{X/Z}$ acts on 
$\oo_{\overline{X}_Z}$, $\widehat{N}$ acts on $\widehat{\oo}_T$ and $\oo_{\overline{T}_Y}={\widehat{\oo}_T}/{\mathcal{E} .\, \widehat{\oo}_T}$
so that for any sheaf $\mathcal{F}$ of $\oo_{\overline{X}_Z}$-modules,
\begin{equation} \label{malin}
\mathcal{F}_{|\,Y}=\frac{\left(\sigma_* \mathcal{F} \right)_{|\,T}}{\mathcal{E} .\left(\sigma_* \mathcal{F} \right)_{|\,T}} \cdot
\end{equation}
Let us now consider the Leray filtration on $\Lambda^k \widehat{N}$. The p-th piece $\M{F}_p$ of this filtration is the image of $\Lambda^p \mathcal{E} \otimes \Lambda^{k-p} \widehat{N}$ in $\Lambda^k \widehat{N}$. Since $\M{Gr}_{\, 1} \left(\Lambda^k \widehat{N} \right)=\Lambda^k N^*_{T/Y}$, it follows from (\ref{malin}) that the restriction of the AK complex $\mathcal{P}_{\sigma^{\vphantom{a}}}$ to $Y$ is given for any nonnegative integer $k$ by the formula
\begin{equation} \label{gr1}
\left( \mathcal{P}_{\sigma^{\vphantom{a}}} \right)_{-k\, | \,Y}=\Lambda^{k+1} N^*_{T/Y} \oplus \Lambda^k \widehat{N}.
\end{equation}
We can describe explicitly $\left( \mathcal{P}_{\sigma^{\vphantom{a}}} \right)_{\, | Y}$ in $\M{D}^{\M{b}}\left(\overline{T}_Y\right)$: the natural map $\M{F}_1 \left( \Lambda^k \widehat{N} \right) \longrightarrow \left( \mathcal{P}_{\sigma^{\vphantom{a}}} \right)_{-k\, | \,Y}$ given via (\ref{gr1}) by the inclusion on the last factor and zero on the first factor is $\oo_{\overline{T}_Y}$-linear. It follows that we have a quasi-isomorphism in $\M{D}^{\M{b}}\left(\overline{T}_Y\right)$:
\begin{equation} \label{yey}
\left( \mathcal{P}_{\sigma^{\vphantom{a}}} \right)_{\, | Y}\simeq\bigoplus_{k \geq 0} \, \M{F}_1  \left( \Lambda^k \widehat{N} \right) \, [k].
\end{equation}
It is important to remark that this isomorphism is \textit{not} induced by a morphism of complexes. 
\section{Formality in the split case}
Let us define a morphism
\[
\Psi_{\sigma} \,\colon \, j_{Y/Z}^*\, j_{X/Z*} \, \oo_{X} \simeq j_{Y/Z}^*\, j_{X/Z*} \, \mathcal{P}_{\sigma} \longrightarrow \left( \mathcal{P}_{\sigma^{\vphantom{a}}} \right)_{\, | Y}\simeq\bigoplus_{k \geq 0} \, \M{F}_1  \left( \Lambda^k \widehat{N} \right) \, [k]
\]
in $\M{D}^{\M{b}}(Y)$ using (\ref{yey}).  If (\ref{excess}) splits, then for any positive integer $k$, any splitting of this sequence yields an isomorphism between $\M{F}_1  \left( \Lambda^k \widehat{N} \right)$ and $\bigoplus_{\ell=1}^k \Lambda^{\ell} \mathcal{E} \otimes \Lambda^{k-\ell} N^*_{T/Y}$.

\begin{proposition} Let $(X, \sigma)$, $Y$, $Z$ be as above; and assume that the excess exact sequence \emph{(\ref{excess})} splits. Then the composite morphism
\[
\Theta_{\sigma} \, \colon \, j_{Y/Z}^*\, j_{X/Z*} \, \oo_{X} \overset{\Psi_{\sigma}}{\longrightarrow} \bigoplus_{k \geq 0} \, \M{F}_1  \left( \Lambda^k \widehat{N} \right) \, [k]
\longrightarrow \mathfrak{s} \left(\mathcal{E}\right)
\]
is an isomorphism in $\M{D}^{\M{b}}(Y)$.
\end{proposition}
\begin{proof}
Let us describe $\Psi_{\sigma}$ locally. For this we use the computations in local coordinates performed at the end of \S \ref{linear}. Taking the same notations, we see that we have a commutative diagram
\[
\xymatrix{
\mathfrak{s} \left( \mathcal{E} \right) \ar[r]^-\sim \ar[d]&  \mathcal{K}(U)_{|Y} \ar[d] \\
\bigoplus_{k \geq 0} \, \M{F}_1  \left( \Lambda^k \widehat{N} \right) \, [k] \ar[r]^-{\sim}& \left(\mathcal{P}_{\tau}\right)_{\, |\,U \cap Y}
}
\]
is an isomorphism in $\M{D}^{\M{b}}(Y \cap U)$, where $\tau$ is the local quantization of $X$ adapted to our choice of local holomorphic coordinates, and the right vertical arrow is given by (\ref{morphisme}). Then we can prove easily that the same results holds if we replace $\sigma$ by $\tau$. Indeed, if $U$ is small enough, we can pick a holomorphic connection on $N^*_{X/Z}$ on $U$; and we have a commutative diagram
\[
\xymatrix{ \bigoplus_{k \geq 0} \, \M{F}_1  \left( \Lambda^k \widehat{N} \right) \, [k] \ar[r]^-\sim \ar[d]_-{\M{id}}&
\left(\mathcal{P}_{\sigma}\right)_{\, |\,U \cap Y}  \ar[d]^{\sim}  \\
\bigoplus_{k \geq 0} \, \M{F}_1  \left( \Lambda^k \widehat{N} \right) \, [k] \ar[r]^-\sim
&\left(\mathcal{P}_{\tau}\right)_{\, |\,U \cap Y} }
\]
where the left vertical isomorphism is given by (\ref{change}). This yields the result.
\end{proof}
Remark that even if the excess sequence (\ref{excess}) if not split, we always an Atiyah morphism
\begin{equation} \label{esoterique}
\mathfrak{at}_{\,\sigma} \, \colon \, j_{Y/Z}^*\, j_{X/Z*} \, \oo_{X} \longrightarrow \mathcal{E}\, [1]
\end{equation}
which is the component of degree $-1$ of $\Psi_{\sigma}$. This morphism induces an isomorphism on the local cohomology sheaves in degree $-1$ of both sides.
%The set of splittings of (\ref{excess}) is an affine space over $\T{Hom}\left(N^*_{T/Y}, \mathcal{E} \right)$. Let us fixed a fixed splitting $\delta$ of (\ref{excess}). If we take an element $\phi$ in $\T{Hom}\left(N^*_{T/Y}, \mathcal{E} \right)$, then the morphism from $\Lambda^k \widehat{N}$ to $\Lambda^k \mathcal{E}$ induced by the splitting $\delta + \phi$ is given via the splitting $\delta$ by   
%\[
%e_1 \wedge \ldots \wedge e_p \wedge n_1 \wedge \ldots \wedge n_{k-p} \longrightarrow e_1 \wedge \ldots \wedge e_p \wedge \phi(n_1) \wedge \ldots \wedge \phi(n_{k-p})
%\]
\section{The reverse implication}
Let us now assume that $j_{Y/Z}^*\, j_{X/Z*}\, \oo_X$ is formal in $\M{D}^{\M{b}}(Y)$. Remark that we have a canonical morphism
\[
j_{Y/Z}^*\, j_{X/Z*}\, \oo_X \longrightarrow \oo_T \oplus \mathcal{E}\,[1]
\]
which is an isomorphism on the local cohomology sheaves of degrees $0$ and $-1$. The first component is given by the composition
\[
\mathfrak{p} \, \colon \, j_{Y/Z}^*\, j_{X/Z*} \, \oo_X \longrightarrow \M{H}^0 \left( j_{Y/Z}^*\right)\, \left\{j_{X/Z*} \left( \oo_X \right)\right\} \simeq j_{T/Y*} \, \oo_T
\]
and the second one is the Atiyah morphism (\ref{esoterique}). Therefore, if $\beta=\bigoplus_{k \geq 0} \beta_k$ is an isomorphism from $j_{Y/Z}^*\, j_{X/Z*} \oo_X$ to $\bigoplus_{k \geq 0} \Lambda^k \mathcal{E} [k]$, we can replace $
\beta_0$ by $\mathfrak{p}$ and $\beta_1$ by $\mathfrak{at}_{\, \sigma}$. At usual, we denote by $\mathcal{P}_{\sigma}$ the AK complex associated with $(X, \sigma)$. Let us consider the morphism
\begin{align*}
\mathfrak{m}& \, \colon \,\mathfrak{s} \left( \mathcal{E} \right)\, \lltens[\oo_Y]\, \mathfrak{s} \left( \mathcal{E} \right) \overset{\beta^{-1}\, \lltens[] \,\beta^{-1}}{\longrightarrow} \Bigl( j_{Y/Z}^*\, j_{X/Z*} \oo_X \Bigr) \, \lltens[\oo_Y] \, \Bigl( j_{Y/Z}^*\, j_{X/Z*} \oo_X \Bigr) \overset{\sim}{\longrightarrow} j_{Y/Z}^*\, \Bigl( \oo_X \lltens[\oo_Z] \oo_X \Bigr)\\
&\simeq j_{Y/Z}^*\, \Bigl( \oo_X \lltens[\oo_Z] \mathcal{P}_{\sigma} \Bigr) \overset{\sim}{\longrightarrow} j_{Y/Z}^*\, \left\{\mathfrak{s} \bigl( \widehat{N}^* \bigr)\right\} \longrightarrow j_{Y/Z}^*\, \left\{\oo_X \oplus \widehat{N}^* \right\} \longrightarrow \oo_T \oplus \mathcal{E}\,[1] \oplus \widehat{N}^*\,[1].
\end{align*}
Note that the last two arrows occurring in the expression of $\mathfrak{m}$ induce isomorphisms on local cohomology sheaves of degrees $0$ and $-1$.
Since $\mathcal{H}^{-1} \bigl( \oo_T\, \lltens[\oo_Y] \oo_T \bigr)$ is canonically isomorphic to $N^*_{T/Y}$, we obtain an isomorphism $
\mathcal{H}^{\,-1} (\mathfrak{m}) \, \colon \, \mathcal{E} \oplus \mathcal{E} \oplus N^*_{T/Y} \overset{\sim}{\longrightarrow} \mathcal{E} \oplus \widehat{N}
$ in $\M{Mod}(Y)$, hence in $\M{Mod}(T)$; we denote it $\mathfrak{n}$. For later use, let us precise that we take for the first factor $\mathcal{E}$ the term $\mathcal{E} \otimes  \oo_T$. Using the natural morphism $\oo_{Z} \longrightarrow j_{X/Z*}\,\oo_{X}$, we get a commutative diagram
\[
\xymatrix{ \oo_Y \lltens[\oo_Y] \, j_{Y/Z}^*\, \mathcal{P}_{\sigma}  \ar[r]^-\sim \ar[d]&
j_{Y/Z}^*\, \Bigl( \oo_Z \lltens[\oo_Z] \,\mathcal{P}_{\sigma} \Bigr)  \ar[d] \\
 j_{Y/Z}^* \,\oo_X  \, \lltens[\oo_Y]\,  j_{Y/Z}^*\, \,\mathcal{P}_{\sigma}  \ar[r]^-\sim
&j_{Y/Z}^*\, \Bigl( \oo_X \lltens[\oo_Z] \,\mathcal{P}_{\sigma} \Bigr) }
\]
which yields another commutative diagram, namely
\[
\xymatrix@R=0.8cm@C=1.6cm{ \mathcal{E} \ar[r]^-{\M{id}} \ar[d]_-{(0, \M{id}, 0)}&
\mathcal{E}  \ar[d]^-{(\M{id}, \alpha)} \\
\mathcal{E} \oplus \mathcal{E} \oplus N^*_{T/Y} \ar[r]^-{\mathfrak{n}}
&  \mathcal{E} \oplus \widehat{N}^*}
\]
We can do the same by switching the factors: using the morphism $\oo_{Z} \longrightarrow \mathcal{P}_{\sigma}$, we get the diagram
\[
\xymatrix{ j_{Y/Z}^*\, \oo_X \,  \lltens[\oo_Y] \oo_Y \ar[r]^-\sim \ar[d]&
j_{Y/Z}^*\, \Bigl( \oo_X \lltens[\oo_Z] \,\oo_Z \Bigr)  \ar[d] \\
 j_{Y/Z}^*\, \oo_X \, \lltens[\oo_Y] \, j_{Y/Z}^*\, \mathcal{P}_{\sigma} \ar[r]^-\sim
&j_{Y/Z}^*\, \oo_X \, \lltens[\oo_Z] \,\mathcal{P}_{\sigma} }
\]
which gives
\[
\xymatrix@R=0.8cm@C=1.6cm{ \mathcal{E} \ar[r]^-{\M{id}} \ar[d]_-{(\M{id}, 0, 0)}&
\mathcal{E}  \ar[d]^-{(\M{id}, 0)} \\
\mathcal{E} \oplus \mathcal{E} \oplus N^*_{T/Y} \ar[r]^-{\mathfrak{n}}
&  \mathcal{E} \oplus \widehat{N}^*}
\]
Therefore wee see that the composite map
\[
\mathcal{E} \oplus N^*_{T/Y}\simeq \{0\} \oplus \mathcal{E} \oplus N^*_{T/Y} \longrightarrow \mathcal{E} \oplus \mathcal{E} \oplus N^*_{T/Y}  \overset{\mathfrak{n}}{\longrightarrow} \mathcal{E} \oplus \widehat{N}^* \longrightarrow \widehat{N}^* 
\]
is an isomorphism whose first component is $\alpha$. This proves that $N^*_{T/Y}$ is a direct factor of $\mathcal{E}$ in $\widehat{N}^*$, so that (\ref{excess}) is holomorphically split.
\section{An alternative proof}
In this section, we sketch an alternative proof of the splitting of the excess exact sequence (\ref{excess}) when $j_{Y/Z}^*\, j_{X/Z*} \, \oo_X$ is formal in $\M{D}^{\M{b}}(Y)$. For the formalism of Grothendieck's six operations in our context, we refer the reader to \cite[\S 5.2]{HKR}.
\par \medskip
Recall that there are \textit{canonical} isomorphisms $\mathcal{H}^{-i} \left(j_{Y/Z}^*\, j_{X/Z*} \, \oo_X \right) \simeq \Lambda^i \mathcal{E}$. Therefore, if the derived intersection $j_{Y/Z}^*\, j_{X/Z*} \, \oo_X$ is formal, we can pick an isomorphism $\varphi$ between $\mathfrak{s} \left(\mathcal{E} \right)$ and $j_{Y/Z}^*\, j_{X/Z*} \, \oo_X$ and we can assume without loss of generality that the induced isomorphisms in local cohomology are the canonical ones.  If $d$ denotes the codimension of $T$ in $Y$, we have
\par \smallskip
\begin{align*}
&\M{Hom}_{\M{D}^{\M{b}}(Y)} \Bigl(j_{T/Y *}\, \mathfrak{s} \left( \mathcal{E} \right), j_{Y/Z}^* \, j_{X/Z*} \, \oo_X \Bigr)\\
\simeq \, \, \, &\M{Hom}_{\M{D}^{\M{b}}(T)} \Bigl(\mathfrak{s} \left( \mathcal{E} \right), j_{T/Y}^{!}\,\{ j_{Y/Z}^* \, j_{X/Z*} \, \oo_X \}\Bigr)\\
\simeq \, \, \, & \M{Hom}_{\M{D}^{\M{b}}(T)} \Bigl(\mathfrak{s} \left( \mathcal{E} \right), j_{T/Y}^{*}\,\{ j_{Y/Z}^* \, j_{X/Z*} \, \oo_X \} \, \lltens[\oo_T] \, j_{T/Y}^{!}\, \oo_Y\Bigr)\\
\simeq \, \, \, & \M{Hom}_{\M{D}^{\M{b}}(T)} \Bigl(\mathfrak{s} \left( \mathcal{E} \right), j_{T/X}^{*}\,\{ j_{X/Z}^* \, j_{X/Z*} \, \oo_X \} \, \lltens[\oo_T] \, j_{T/Y}^{!}\, \oo_Y\Bigr)\\
\simeq \, \, \, & \M{Hom}_{\M{D}^{\M{b}}(T)} \Bigl(\mathfrak{s} \left( \mathcal{E} \right), j_{T/X}^{*}\,\{\mathfrak{s} \left( N^*_{X/Z} \right) \} \, \lltens[\oo_T] \, \Lambda^d N_{T/Y} [-d]\Bigr)\\
\simeq \, \, \, & \M{Hom}_{\M{D}^{\M{b}}(T)} \Bigl(\mathfrak{s} \left( \mathcal{E} \right), \mathfrak{s} \bigl( \widehat{N}^*\bigr)\, \otimes \, \Lambda^d N_{T/Y} [-d]\Bigr)\\
\simeq \, \, \, & \bigoplus_{q-p\geq d} \M{Ext}^{q-p-d} \Bigl(\Lambda^p \mathcal{E},  \Lambda^q \widehat{N}^* \otimes \Lambda^d N_{T/Y} \Bigr)\\
\simeq \, \, \, & \bigoplus_{q-p\geq d} \M{Ext}^{q-p-d} \Bigl(\Lambda^p \mathcal{E} \otimes \Lambda^d N^*_{T/Y} , \Lambda^q \widehat{N}^* \Bigr).
\end{align*}
\par \smallskip
Using this decomposition, we write $\varphi=\sum_{p, q} \varphi_{p,q}$. Then for any nonnegative integer $p$ the morphism $\mathcal{H}^{-p}(\varphi)$ is determined by $\varphi_{p,\,p+d}$. Besides, it is possible to check using local holomorphic coordinates that for any $p$, the map $\varphi_{p, \,p+d}$ fits into the commutative diagram below:
\[
\xymatrix{ 
\M{F}_{p} \left(\Lambda^{p+d} \widehat{N}^* \right) \ar[d] \ar[r]
&\Lambda^{p+d} \widehat{N}^*\\
\M{Gr}_{p} \left(\Lambda^{p+d} \widehat{N}^* \right) \ar[ru]_-{\varphi_{p, \,p+d}}
&  }
\]
The reason for this is that we have normalised the action of the isomorphism $\varphi$ on the local cohomology sheaves. If $r$ denotes the rank of the excess bundle $\mathcal{E}$, then transpose of the map
\[
\mathcal{E}^*  \simeq  \Lambda^{r-1} \mathcal{E}  \otimes  \Lambda^r \mathcal{E}^*  \simeq   \bigl(\Lambda^{r-1} \mathcal{E}  \otimes  \Lambda^d N_{T/Y}^* \bigr)  \otimes  \Lambda^{r+d} \widehat{N} \xrightarrow{\varphi_{r-1, r-1+d} \,\otimes\, \M{id}} \Lambda^{r+d-1} \widehat{N}^*  \otimes   \Lambda^{r+d} \widehat{N}  \simeq \widehat{N}
\]
is a retraction of $\alpha$ qed.
\par \bigskip
This method allows to provide pathological examples:
\begin{proposition}
Assume that the natural surjection $\Lambda^d \widehat{N} ^* \otimes \Lambda^d N_{T/Y} \longrightarrow \oo_T$ does not split. Then $j_{Y/Z}^*\, j_{X/Z*} \, \oo_X$ does not live in the image of $\M{D}^{\M{b}}(T)$ in $\M{D}^{\M{b}}(Y)$.
\end{proposition}
\begin{proof} Assume that there exist an element  $\theta$ in $\M{D}^{\M{b}}(T)$ such that $j_{Y/Z}^*\, j_{X/Z*} \, \oo_{X}=j_{T/Y*}\, \theta$. There is a canonical morphism $\oo_Y \simeq j_{Y/Z}^* \, \oo_Z \longrightarrow j_{Y/Z}^*\, j_{X/Z*} \, \oo_{X}$ inducing the natural morphism $\oo_Y \longrightarrow \oo_T$ on the local cohomology sheaves of degree $0$. Then using the chain
\[
\M{Hom}_{\M{D}^{\M{b}}(Y)} \Bigl(\oo_Y, j_{T/Y*} \, \theta \Bigr) \simeq \M{Hom}_{\M{D}^{\M{b}}(T)} \Bigl( j_{T/Y}^* \,\oo_Y, \theta \Bigr) \longrightarrow \M{Hom}_{\M{D}^{\M{b}}(Y)} \Bigl( j_{T/Y*}\, \oo_T, j_{Y/Z}^*\, j_{X/Z*} \, \oo_{X} \Bigr)
\]
we get a morphism from $\oo_T$ to $j_{Y/Z}^*\, j_{X/Z*} \, \oo_{X}$ inducing the canonical isomorphism in local cohomology of degree $0$. Thanks to the previous calculation:
\[
\M{Hom}_{\M{D}^{\M{b}}(Y)} \Bigl( j_{T/Y*}\, \oo_T, j_{Y/Z}^*\, j_{X/Z*} \, \oo_{X} \Bigr) \simeq \bigoplus_{q \geq d} \M{H}^{q-d} \bigl( \Lambda^{q} \widehat{N}^* \otimes \Lambda^{d} N_{T/Y} \bigr).
\]
The component in $\M{H}^{0} \bigl( \Lambda^{d} \widehat{N}^* \otimes \Lambda^{d} N_{T/Y} \bigr)$ must map to $1$ via the map $\Lambda^{d} \widehat{N}^* \otimes \Lambda^{d} N_{T/Y} \longrightarrow \oo_T$, which contradicts our assumption.
\end{proof}

\section{The formality criterion for quantization}
In this section, we prove Theorem \ref{QC}. This is in fact a straightforward consequence of Theorem \ref{boss}. Indeed, if $X$ is a closed submanifold of $Y$, the cycles $\Delta_Y$ and $X^2$ are smooth submanifolds of $Y^2$ intersecting linearly along $\Delta_X$. Remark that $\Delta_Y$ can be obviously quantized (for instance by one of the two projections), and that the excess exact sequence (\ref{excess}) on $\Delta_X$ is the conormal exact sequence
\[
0 \longrightarrow N^*_{X/Y} \longrightarrow \Omega^1_{Y |X} \longrightarrow \Omega^1_X \longrightarrow 0.
\]
Therefore Theorem \ref{boss} yields the result. To give a clearer picture of this result, we can draw the following diagram:

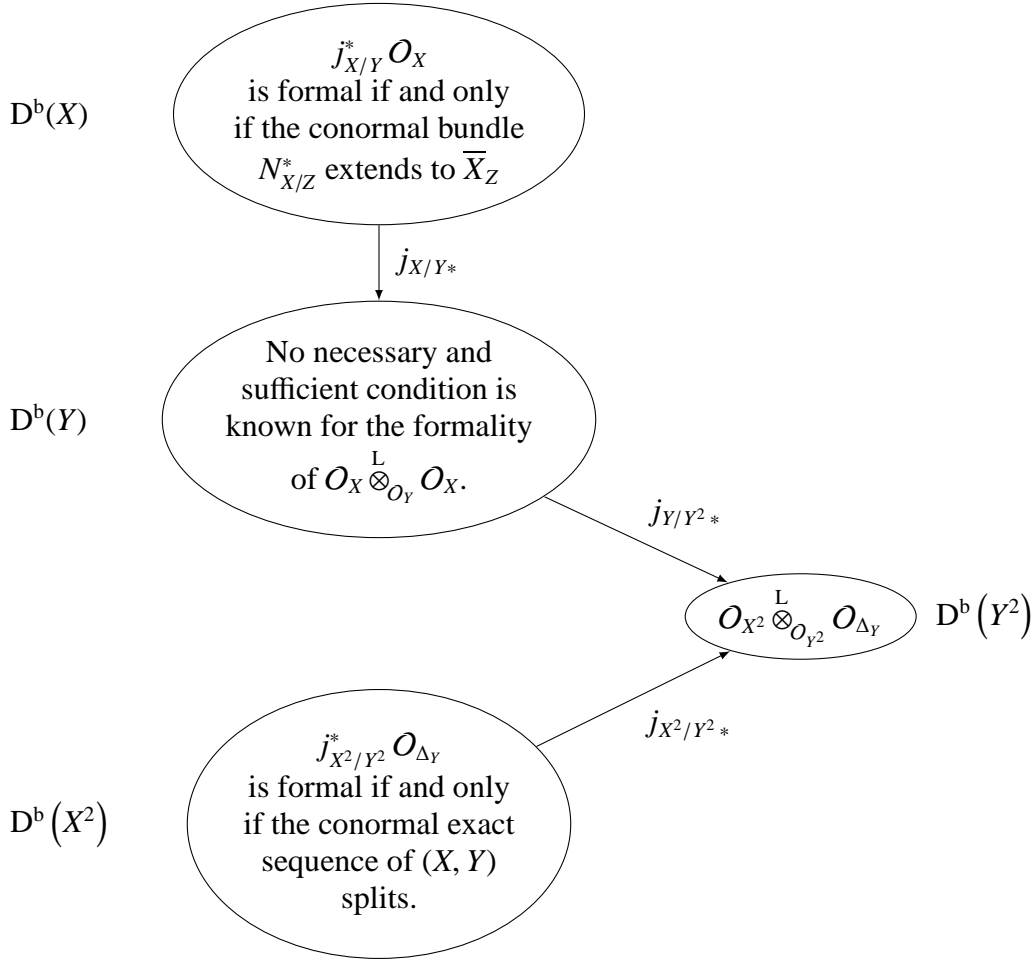
\begin{center}
\begin{tikzpicture}[>=latex,inner sep=0pt]
\node[above left,ellipse,draw] (bulle haute) at (0cm,1.5cm) {\begin{tabular}{@{}c@{}}
             No necessary and \\
             sufficient condition is\\
             known for the formality \\
             of $\oo_{X} \, \lltens[\oo_{Y}] \, \oo_{X}$. \\
             \end{tabular}};
\node[above=of bulle haute,ellipse,draw] (bulle tres haute) {\begin{tabular}{@{}c@{}}
             $j_{X/Y}^* \, \oo_X$  \\
              is formal if and only  \\
              if the conormal bundle \\
              $N^*_{X/Z}$ extends to $\overline{X}_Z$\\
             \end{tabular}};
\node[below=of bulle haute,yshift=-1cm,ellipse,draw] (bulle basse) {\begin{tabular}{@{}c@{}}
             $j_{X^2/Y^2}^* \, \oo_{\Delta_{Y}}$\\
              is formal if and only \\
              if the conormal exact \\
              sequence of $(X,Y)$ \\
          splits.\\
             \end{tabular}};
\node[right,ellipse,draw] (bulle droite) at (2cm,0cm) {\begin{tabular}{@{}c@{}}
             $\oo_{X^2} \, \lltens[\oo_{Y^2}] \, \oo_{\Delta_Y}$\\
             \end{tabular}};
\node[right,xshift=-2cm] at ($(bulle haute.west)!(bulle tres haute)!(bulle haute.west)$) {$\M{D}^{\M{b}}(X)$};
\node[right,xshift=-2cm] at (bulle haute.west) {$\M{D}^{\M{b}}(Y)$};
\node[right,xshift=-2cm] at ($(bulle haute.west)!(bulle basse)!(bulle haute.west)$) {$\M{D}^{\M{b}} \left(X^2 \right)$};
\node[right=of bulle droite,xshift=-0.75cm] {$\M{D}^{\M{b}} \left(Y^2 \right)$};
\draw[->] (bulle tres haute) -- (bulle haute) node[pos=0.5,right=0.2cm] {$j_{X/Y*}$};
\draw[->] (bulle haute) -- (bulle droite) node[pos=0.5,above right=0.2cm] {$j_{Y/Y^2\,*}$};
\draw[->] (bulle basse) -- (bulle droite) node[pos=0.5,below right=0.2cm] {$j_{X^2/Y^2\,*}$};
\end{tikzpicture}
\end{center}
\section{Perspectives}
In view of the preceding discussion, we propose a very natural conjecture, which predicts the formality of \textit{all} derived self-intersections in the bounded derived category of the ambient complex manifold.
\begin{conjecture} For any pair $(X, Y)$, the objet $\oo_{X} \, \lltens[\oo_{Y}] \, \oo_{X}$ is always formal in $\M{D}^{\M{b}}(Y)$.
\end{conjecture}
As an evidence for this conjecture, we provide the following formality result for truncated objects of length two (which corresponds in classical terms to the degeneracy at the page $\M{E}_2$ of a spectral sequence): 
\begin{proposition}
For any nonnegative integer $p$, $\tau^{\geq-(p+1)} \, \tau^{\leq -p} \,\Bigl(\oo_{X} \, \lltens[\oo_{Y}] \, \oo_{X}\Bigr)$ is formal in $\M{D}^{\M{b}}(Y)$.
\end{proposition}
\begin{proof} Let $j=j_{X/Y}$ and $\vartheta=j^* \, j_{*} \, \oo_X$. The Atiyah morphism $\mathfrak{at}\, \colon \, j_* \,\oo_X \longrightarrow j_* N^*_{X/Y}[1]$ in the derived category $\M{D}^{\M{b}}(Y)$ yields by adjunction a morphism $\vartheta \longrightarrow N^*_{X/Y}[1]$, which induces an isomorphism in local cohomology of degree $-1$. Therefore $\tau^{\geq -1} \vartheta \simeq \oo_X \oplus N^*_{X/Y}[1]$ in $\M{D}^{\M{b}}(X)$. We can go a step further in $\M{D}^{\M{b}}(Y)$: let us consider the morphism
\[
j_* \, \vartheta \simeq \oo_{X} \,  \lltens[\oo_{Y}] \, \oo_{X} \xrightarrow{\mathfrak{at} \,\otimes \,\mathfrak{at}} N^*_{X/Y}[1] \,  \lltens[\oo_{Y}] \, N^*_{X/Y}[1] \xrightarrow{\, \, \, \, \wedge \, \, \, \,} \Lambda^2 N^*_{X/Y}[2].
\]
It induces an isomorphism in local cohomology of degree $-2$. Therefore 
\[
\tau^{\geq -2} j_*\, \vartheta \simeq \oo_X \oplus N^*_{X/Y}[1] \oplus \Lambda^2 N^*_{X/Y}[2]
\]
in $\M{D}^{\M{b}}(Y)$. Recall now that $j_* \, \vartheta$ is a ring object in $\M{D}^{\M{b}}(Y)$, the ring structure being given by the composition of morphisms
\[
j_* \,\vartheta \,  \lltens[\oo_{Y}] \,j_*\, \vartheta \longrightarrow j_*\, (\vartheta \,  \lltens[\oo_{X}] \, j^* \, j_*\, \vartheta) \longrightarrow j_*\, (\vartheta \,  \lltens[\oo_{X}] \, \vartheta) \simeq j_*\, j^*\, \Bigl(\oo_X \,  \lltens[\oo_{Y}] \, \oo_X \Bigr) \longrightarrow j_*\, \vartheta.
\]
Let us consider the morphism $\bigl( \tau^{\leq -1} j_* \vartheta \bigr)^{\otimes p} \longrightarrow \bigl(j_* \vartheta \bigr)^{\otimes p} \longrightarrow j_* \, \vartheta$ where the last arrow is the ring multiplication. Since the left-hand side is concentrated in degrees at most $-p$, this morphism can be factored through $\tau^{\leq -p} j_*\, \vartheta$. Applying the truncation functor $\tau^{\geq -{p+1}}$, we get a morphism
\[
\tau^{\geq -{p+1}} \bigl( \tau^{\leq -1} j_* \vartheta \bigr)^{\otimes p} \longrightarrow \tau^{\geq -{p+1}} \tau^{\leq -p} j_*\, \vartheta.
\]
Remark now that the natural morphism $\tau^{\geq -{p+1}} \bigl( \tau^{\leq -1} j_* \vartheta \bigr)^{\otimes p}              \longrightarrow
\tau^{\geq -{p+1}} \bigl( \tau^{\geq -2} \tau^{\leq -1} j_* \vartheta \bigr)^{\otimes p} $ is an isomorphism, so that we obtain a morphism
\[
\theta \, \colon \, \tau^{\geq -{p+1}} \bigl( \tau^{\geq -2}\, \tau^{\leq -1} \,j_* \,\vartheta \bigr)^{\otimes p} \longrightarrow \tau^{\geq -{p+1}} \tau^{\leq -p} j_*\, \vartheta.
\]
We have seen that $\tau^{\geq -2}\, \tau^{\leq -1}\, j_*\, \vartheta$ is formal; this yields
\[
\tau^{\geq -{p+1}} \bigl( \tau^{\geq -2}\, \tau^{\leq -1}\, j_*\, \vartheta \bigr)^{\otimes p} \simeq \Bigl(N^*_{X/Y}[1]\Bigr)^{\otimes p} \oplus \left\{ \Bigl(N^*_{X/Y}[1]\Bigr)^{\otimes p-1} \otimes \Lambda^2\, N^*_{X/Y} [2] \right\}.
\]
Besides, in local cohomology, $\theta$ is simply the wedge-product. Using the antisymmetrization morphisms (see \cite[\S 2.1]{HKR}), we obtain a quasi-isomorphism
\[
\Lambda^p \, N^*_{X/Y} [p] \oplus \Lambda^{p+1}\, N^*_{X/Y} [p+1] \longrightarrow \tau^{\geq -{p+1}} \bigl( \tau^{\geq -2} \tau^{\leq -1} j_* \vartheta \bigr)^{\otimes p} \xrightarrow{\, \, \, \theta \, \, \, } \tau^{\geq -{p+1}} \tau^{\leq -p} j_*\, \vartheta
\]
\end{proof}
\begin{remark} The crucial point in this proof is that $\tau^{\geq -2} \tau^{\leq -1} \,j_* \,\vartheta$ is formal in $\M{D}^{\M{b}}(Y)$. However, $\tau^{\geq -2} \tau^{\leq -1} \,\vartheta$ is not always formal in $\M{D}^{\M{b}}(X)$. In fact, this object is formal if and only if $N^*_{X/Y}$ extends at the first order (see \cite{AC}). This reflects the fact that $j_*$ is not faithful.
\end{remark}
%\nocite{*}
\newpage
\bibliographystyle{plain}
\bibliography{bib}

\end{document}